\documentclass[11pt]{article}
\usepackage[utf8]{inputenc}

\usepackage[numbers]{natbib}
\usepackage{hyperref}
\usepackage{amsmath}
\usepackage{amssymb}
\usepackage{amsthm}
\usepackage{blindtext}
\usepackage[margin = 1 in]{geometry}
\usepackage{mathtools}
\usepackage{graphicx}
\usepackage{tikz-cd}
\usepackage{spectralsequences}
\usepackage{MnSymbol}

\newtheorem{theorem}{Theorem}[section]
\newtheorem{proposition}[theorem]{Proposition}
\newtheorem{corollary}[theorem]{Corollary}
\newtheorem{lemma}[theorem]{Lemma}

\theoremstyle{remark}
\newtheorem*{remark}{Remark}
\newtheorem{definition}[theorem]{Definition}

\usepackage{sectsty}
\sectionfont{\fontsize{12}{15}\selectfont}

\title{Seiberg-Witten Equations and Einstein Metrics on Finite Volume 4-Manifolds with Asymptotically Hyperbolic Ends}
\author{Alex Xu\footnote{Email address: \href{mailto:axu@math.columbia.edu}{axu@math.columbia.edu}}}
\date{}

\begin{document}
\maketitle

\begin{abstract}
    We construct infinitely many examples of finite volume 4-manifolds with $T^3$ ends that do not admit any cusped asymptotically hyperbolic Einstein metrics yet satisfy a strict logarithmic version of the Hitchin-Thorpe inequality due to Dai-Wei. This is done by using estimates from Seiberg-Witten theory due to LeBrun as well as a method for constructing solutions to the Seiberg-Witten equations on noncompact manifolds due to Biquard. We also use constructions coming from the $Pin^-(2)$ monopole equations to obtain a larger class of manifolds where these techniques apply. 
\end{abstract}

\section{Introduction}
A Riemannian manifold $(X,g)$ is said to be Einstein if its Ricci curvature is constant multiple of its metric tensor
\begin{equation*}
    Ric_g = \lambda g.
\end{equation*}
The question of which manifolds admit such metrics has been studied extensively \cite{Besse}. Not every 4-manifold admits an Einstein metric. A necessary condition for a closed orientable 4-manifold $X$ to admit Einstein metrics is that it must satisfy the Hitchin-Thorpe inequality \cite{Hitchin74}
\begin{equation*}
    2\chi(X) - 3\lvert \sigma(X)\rvert \ge 0
\end{equation*}
with equality if and only if $X$ is flat or Calabi-Yau. Here $\chi$ and $\sigma$ denote the Euler characteristic and signature, respectively. Satisfying the Hitchin-Thorpe inequality is a necessary, but not sufficient condition for $X$ to admit an Einstein metric. Using techniques from Seiberg-Witten theory, LeBrun \cite{LeBrun96} constructed the first examples of closed 4-manifolds that satisfy the strict Hitchin-Thorpe inequality but nonetheless do not admit Einstein metrics. Kotschick \cite{Kotschick98} and Braungardt-Kotschick \cite{BraungardtKotschick05} showed that existence of Einstein metrics depends on the smooth structure of the underlying manifold by constructing pairs $(X_i,Z_i)$ of closed oriented 4-manifolds with $X_i$ homeomorphic to $Z_i$ and so that $X_i$ admits K\"{a}hler-Einstein metric and so that $Z_i$ does not admit any Einstein metric. See also the works of Del Rio Guerra \cite{DelRioGuerra}, LeBrun \cite{LeBrun01}, and Ishida \cite{Ishida12} for nonexistence of Einstein metrics in the non simply connected case. 

This paper is interested in a finite volume generalization of this problem. Suppose that $X$ is a noncompact oriented Riemannian 4-manifold of finite topological type with cylindrical ends all of the form $T^3 \times [0,\infty)$. Such manifolds arise as link complements $\overline{X} - L$, where $\overline{X}$ is ambient closed oriented manifold and $L$ is a disjoint collection of smoothly embedded 2-tori with each component $T_i$ having self-intersection 0. Conversely, via generalized Dehn filling, we can realize all noncompact oriented Riemannian 4-manifold of finite topological type with $T^3$ ends as link complements in some ambient closed oriented 4-manifold. Given a flat metric $g_{T^3}$ on $T^3$, one can construct a hyperbolic cusp metric on $T^3 \times [0,\infty)$ by taking
\begin{equation*}
    g_{hyp} = dt^2 + e^{-2t}g_{T^3}
\end{equation*}
where $t$ is the coordinate on $[0,\infty) $. Say that a metric $g$ on $X$ is asymptotically hyperbolic if on each end there is some flat metric $g_{T^3}$ so that $g$ is asymptotically $C^2$-close to $g_{hyp}$. 

The most obvious source of finite volume asymptotically hyperbolic Einstein metrics comes from hyperbolic 4-manifolds. By Long-Reid \cite{LongReid00} all finite volume hyperbolic 4-manifolds with $T^3$ cusps will have signature 0, so it is an interesting problem to construct examples of noncompact 4-manifolds with signature 0 that do not admit any asymptotically hyperbolic cusps. Examples of hyperbolic link complements of simply connected 4-manifolds have been constructed by Invansic\cite{Ivansic04}, Invansic-Dubravko-Ratcliffe \cite{IvansicDubravkoRatcliffe05}, and Saratchandran \cite{Saratchandran18}.

Another rich source of examples of asymptotically hyperbolic Einstein metrics was constructed by Anderson \cite{Anderson06} via generalized Dehn filling of finite volume hyperbolic 4-manifolds. See also \cite{Bamler12}. To describe the construction, let $X$ be a finite volume hyperbolic 4-manifold with $q$ cusp ends, all isometric to a standard hyperbolic cusp metric $T^3_k \times [0,\infty)$. For $p \le q$ of the cusp ends of $X$, choose a simple geodesic $\gamma_k \subset T^3_k$; implicitly, we have a flat metric on $T^3_k$ coming from restriction of the hyperbolic cusp. On each of the $p$ ends, attach a solid torus $D^2 \times T^2$ to $T^3_k$ by a diffeomorphism $\partial D^2 \times T^2 \cong T^3$ sending $\partial D^2$ to $\gamma_k$. Denote $\overline{\gamma} = (\gamma_1,\cdots, \gamma_p)$ and let $X_{\overline{\gamma}}$ be the manifold obtained by Dehn filling on all of the ends. When $p < q$, $X_{\overline{\gamma}}$ will be noncompact, with $p - q$ cusps. Anderson showed that if the length of each $\gamma_k$ is long enough, then $X_{\overline{\gamma}}$ admits an Einstein metric, and the metric on the cusps will be asymptotically hyperbolic. This is done by constructing an approximate Einstein metric on the Dehn-filling and using an inverse function theorem to perturb the metric to an actual Einstein metric. Note that all such examples are necessarily signature 0; to the best of the author's knowledge, it is still unknown whether there exist any manifolds admitting asymptotically hyperbolic Einstein metrics with only $T^3$ cusps and nonzero signature.

The goal of this paper is to find an obstruction to the existence of finite volume asymptotically hyperbolic Einstein metrics. In this setting, Dai-Wei \cite{DaiWei07} have proved a generalization of the Hitchin-Thorpe inequality for complete finite volume manifolds with with a specified asymptotic geometry at infinity. The version of their theorem adapted to complete, asymptotically hyperbolic Einstein metrics with $T^3$  ends is the following. 
\begin{theorem}[Dai-Wei]\label{DaiWei}
Suppose that $(X,g)$ is a complete noncompact Einstein 4-manifold with cylindrical ends all of the form $T^3 \times [0,\infty)$. Furthermore, suppose that $g$ is an asymptotically hyperbolic metric. Then $X$ satisfies
\begin{equation*}
    2\chi(X) \ge 3\left\lvert \sigma(X)\right\rvert.
\end{equation*}
\end{theorem}
The main theorem of this paper allows us to construct 4-manifolds $X$ with $T^3$ cusps that satisfy the strict inequality of Dai-Wei yet do not admit any asymptotically hyperbolic Einstein metrics. 
\begin{theorem}\label{MainTheorem}
    Let $X_1$ be a closed symplectic manifold with $b^+(X_1) \ge 2$, and let 
    \begin{equation*}
        X_2 = \#_i^k (S^1 \times Y_i) \#_j^{\ell} (S^2 \times \Sigma_{g_j})
    \end{equation*}
    Let $\overline{X} = X_1 \# X_2$ and let $L$ be a collection of smoothly embedded 2-tori each with 0 self-intersection. If $ -\chi(X_2) + 2 \ge \frac{1}{3} (2 \chi(X_1) + 3\sigma(X_1))$, then it follows that $X = \overline{X} - L$ does not admit any asymptotically hyperbolic Einstein metrics. 
\end{theorem}
\begin{remark}
    We can construct construct examples of manifolds that do not admit any asymptotically hyperbolic Einstein metrics yet satisfy the inequality in theorem \ref{DaiWei} by choosing $X_2$ that satisfies the inequality $\frac{1}{2}  (2 \chi(X_1) + 3\sigma(X_1)) > -\chi(X_2) + 2 $.
\end{remark}
This allows us to construct many examples of 4-manifolds that do not admit any asymptotically hyperbolic Einstein metrics by using results on the geography of symplectic 4-manifolds and general type surfaces. Let $(x,c)$ be a pair of positive integers such that  $c \equiv x \mod 8$. For all but finitely many such pairs satisfying $c < 9 x$, Park \cite{Park07} has constructed symplectic 4-manifolds $X$ with $ \chi_h(X) = x$ and $c_1^2(X) = c$. The fundamental group of these manifolds can be chosen to be any finitely generated group. Here $\chi_h$ is the holomorphic Euler characteristic, defined by
\begin{equation*}
   \chi_h(X) = \frac{\sigma(X) + \chi(X)}{4}.
\end{equation*}
Because for a symplectic 4-manifold $c_1^2(X) = 2\chi(X) + 3\sigma(X)$, it follows that for all but finitely many pairs $(\chi,\sigma)$ of integers satisfying $\chi > 0$, $ -2\chi/3 \le \sigma < \chi/3$, and $\chi + \sigma \equiv 0 \mod 4$ there exists a simply connected spin symplectic 4-manifold $X$ with $(\chi(X),\sigma(X)) = (\chi,\sigma)$. 

Now fix a pair of integers $(\chi,\sigma)$ with $\chi + \sigma \equiv 0\mod 2$. We construct construct a closed oriented 4-manifold $\overline{X} = X_1 \# X_2$ satisfying the hypothesis of theorem \ref{MainTheorem} with $\chi(\overline{X}) = \chi$ and $\sigma(\overline{X}) = \sigma$ as follows. Choose $X_1$ following Park's construction with $\sigma(X_1) = \sigma$ and $\chi(X_1) \ge 3\chi + 9\lvert \sigma \rvert$, and let $X_2$ be a connect sum of components of the form $Y \times S^1$ and $\Sigma_g \times S^2$ so that $\chi(X_2) = \chi + 2 - \chi(X_1)$. Since $- \chi \le \chi(X_1)/3$, it follows that 
\begin{equation*}
    \begin{split}
        -\chi(X_2) + 2 &= \chi(X_1) - \chi\\
        &\ge \frac{1}{3}(2\chi(X_1) + 3 \lvert \sigma(X_1) \rvert)
    \end{split}
\end{equation*}
and therefore $\overline{X}$ satisfies the hypothesis of theorem \ref{MainTheorem}. So any complement of smoothly embedded 2-tori in $\overline{X}$, each with 0 self-intersection, will not admit any cusped asymptotically hyperbolic Einstein metrics. 

We can also find many examples coming from general type complex surfaces. Given $X$ a surface of general type, define the Chern slope to be $c_1^2(X) / c_2(X)$. Persson \cite{Persson81} showed that every rational number $r \in [1/5,2]$ arises as the Chern slope of some simply connected general type surface. Rolleau and Urz\'{u}a \cite{RoulleauUrzua15} showed that the Chern slopes of simply connected complex surfaces of general type are dense in the interval $[2,3]$. When $G$ is the fundamental group of nonsingular complex projective surface, Troncoso and Urz\'{u}a \cite{TroncosoUrzua23} show that Chern slopes of complex surfaces of general type with fundamental group $G$ are dense in $[1,3]$.

For more explicit examples of manifolds satisfying the hypothesis of theorem \ref{MainTheorem}, consider the following. Given $n \ge 5$, Baykur and Hamada \cite{BaykurHamada23} have recently constructed minimal symplectic manifolds $X_n$ that are homeomorphic to $\#^{2n+1}(S^2 \times S^2)$. It follows that for $g \ge (4n + 7)/6$, the manifold 
\begin{equation*}
    \overline{X} = X_n\# (S^2 \times \Sigma_g)
\end{equation*}
satisfies the hypothesis of theorem \ref{MainTheorem}. 

Einstein metrics in the finite volume noncompact setting have been studied by Biquard \cite{Biquard97} for finite volume quotients of $\mathbb{CH}^2$ and Di Cerbo \cite{DiCerbo11}\cite{DiCerbo13} in the case where $X$ is a finite volume K\"{a}hler surface of logarithmic general type. We now briefly sketch the idea of the proof; the strategy of the proof follows that of Biquard and Di Cerbo. Suppose that $X$ admits an asymptotically hyperbolic Einstein metric. In the proof of theorem \ref{DaiWei}, Dai and Wei derive a scalar curvature lower bound for asymptotically hyperbolic Einstein metrics
\begin{equation}\label{DaiWeiEstimate}
        2\chi(X) - 3 \lvert \sigma(X) \rvert \ge \frac{1}{4\pi^2}\int_X \frac{s_g^2}{24} dvol_g.
\end{equation}
Conversely, suppose that $g$ is an asymptotically hyperbolic metric on $X$ that admits an $L^2$ solution  $(A,\phi)$ to the Seiberg-Witten equations. Then, there an $L^2$ analogue of an estimate of LeBrun \cite{LeBrun95} proved in the complete finite volume setting by Di Cerbo \cite{DiCerboThesis}. The statement of this theorem uses the language of $L^2$ cohomology; for more details see Section \ref{L2cohomology}. 
\begin{theorem}[Di Cerbo]\label{DiCerboEstimate}
    Suppose that $(X,g)$ is a closed 4-manifold and suppose that $(A,\psi)$ is an $L^2$ solution to the Seiberg-Witten equations. Then
    \begin{equation}
        \frac{1}{32 \pi^2}\int_X s_g^2 dvol_g \ge c_1^2(\mathfrak{s})
    \end{equation}
    where $c_1(\mathcal{L})$ is the $L^2$ Chern class of $[F_A]$. Equality holds if and only if $g$ has constant scalar curvature and is K\"ahler with respect to some complex structure.
\end{theorem}
Combining these two inequalities, we see that if $(X,g)$ is an asymptotically hyperbolic Einstein manifold that admits an $L^2$ solution  $(A,\phi)$ to the Seiberg-Witten equations, then 
\begin{equation}
    2\chi(X) - 3 \lvert \sigma(X) \rvert \ge \frac{1}{3}c_1^2(\mathfrak{s})
\end{equation}
So, in order to show nonexistence of an asymptotically hyperbolic Einstein metric, it is sufficient to prove an existence theorem for $L^2$ solutions of the Seiberg-Witten equations for any asymptotically hyperbolic metric $g$ with $\frac{1}{3}c_1^2(\mathfrak{s}) > 2\chi(X) - 3 \lvert \sigma(X)\rvert$.

To construct such Seiberg-Witten monopoles, we follow a strategy due to Biquard \cite{Biquard97}. We construct a sequence of metric compactifications $(\overline{X},g_j)$ that converge in the $C^\infty$ topology on compact sets to $(X,g)$. Let $(A_j,\phi_j)$ be solutions to the Seiberg-Witten equations on $(\overline{X},g_j, \mathfrak{s})$. By carefully controlling the curvature of the pinching process we can perform Seiberg-Witten bootstrap over compact subsets $K \subset X$ to get a limiting solution $(A,\phi)$ on $(X,g)$. The metric compactifications are the main content of Section \ref{MetricApprox}. 

\subsection{Acknowledgements}
I would like to thank my advisor Francesco Lin for suggesting this problem and for his guidance and encouragement throughout the entire process. I would also like to thank John Morgan, Shuang Liang, Hokuto Konno, and Masaki Taniguchi for helpful conversations over the course of this project. The author was partially supported by NSF grant DMS-2203498.

\section{Approximating Metrics and Pertubations}\label{MetricApprox}
In this section we construct the approximating metrics $g_j$ on $\overline{X}$. Let $X = \overline{X} - L$, and suppose that $L$ is a finite disjoint union of smoothly embedded 2-tori, each with 0 self-intersection, so that $X$ has $T^3$ ends, and suppose that $g$ is an asymptotically hyperbolic metric on $X$.  For a given end $T^3 \times [0,\infty)$, suppose that with respect to a given flat metric $g_0$ on $T^3$, $g$ is asymptotically $C^2$ close to the model hyperbolic metric $dt^2 + e^{-2t}g_0$. Decompose $T^3 = S^1 \times T^2$, where the $T^2$ component comes from embedded 2-torus and the $S^1$ comes from the unit normal bundle, and write $g_0 = g_{T^2} + g_{S^1}$.

We define warping functions $\phi_j$ and $\psi_j$ as follows. Let $\phi_j(t) = e^{-t}$ for $t\in [0,j]$. For $\varepsilon_j$ small we let $\phi_j'$ very rapidly decrease to $-1$ on the interval $[j,j+ \varepsilon_j]$ by making $\phi_j''$ very negative, and for an appropriate choice of $T_J$, let $\phi_j(t) = T_j -t$ for $t \in [j+  \varepsilon_j,T_j]$. Similarly, define $\psi_j$ with $\psi_j(t) = e^{-t}$ for $t \in [0,j]$ and $\psi_j(t) = C_j$ for an appropriately chosen constant $C_j$ on the interval $t \in [j + \varepsilon_j,T_j)$. In the intermediate interval  $[j,j+ \varepsilon_j]$, we slowly increase $\psi_j'$ and decrease $\psi_j''$ to 0 such that $\lvert \psi'' \rvert \le \lvert \psi' \rvert$. This is always possible for $\varepsilon_j$ small enough and $C_j$ close enough to $e^{-j}$. We furthermore also impose the conditions that
\begin{equation}\label{ApproxAssumptions}
\begin{split}
    \frac{\phi_j'}{\phi_j} + 2\frac{\psi_j'}{\psi_j} \le -3
    \frac{\psi_j'}{\psi_j} \ge \frac{-2}{1 - \frac{\phi_j'}{\phi_j}}
\end{split}
\end{equation}
These inequalities will be important for curvature estimates. On the interval $[0,T_j) \times T^3$ define the metric
\begin{equation}
    \Tilde{g}_j = dt^2 + \psi_j^2 g_{T^2} + \phi_j^2 g_{S^1}.
\end{equation}
Since $\phi_j(T_j) = 0$ follows that the $S_1$ factor is completely pinched off, so $\Tilde{g}_j$ extends to the compactification $D^2 \times T^2$. The metric $\Tilde{g}_j$ is a standard hyperbolic cusp on the interval $[0,j]$ and a standard flat metric on the interval $[j+ \varepsilon_j,T_j]$. The condition on $\psi_j'/\psi_j$ is to ensure that the mean curvature of the slices in $\Tilde{g}_j$ is at least $3/2$. Indeed, we note that the following inequalities are true for all $t$ and $j$.
\begin{equation}\label{ApproxEstimates}
\begin{split}
    \frac{-\partial_t^2 \phi_j}{\phi_j} - 2\frac{\partial^2_t \psi_j}{\psi_j} + \frac{(\partial_t\phi_j)^2}{\phi_j^2} + 2\frac{(\partial_t\psi_j)^2}{\psi_j^2} &\ge -2\\
    \frac{-\partial_t^2 \phi_j}{\phi_j} - 2\frac{\partial_t \phi_j}{\phi_j} \cdot \frac{\partial_t \psi_j}{\psi_j} + \frac{(\partial_t\phi_j)^2}{\phi_j^2} &\ge -2\\
    \frac{-\partial^2_t \psi_j}{\psi_j} - \frac{\partial_t \phi_j}{\phi_j} \cdot \frac{\partial_t \psi_j}{\psi_j} &\ge -2
\end{split}
\end{equation}
These will be used in section \ref{SeibergWitten} to study the convergence of solutions of Seiberg-Witten equations.

Fix a smooth cutoff function $\chi$ with $\chi(t) = 0$ for $t \ge 1$ and $\chi(t) = 1$ for $t \le 0$. Set $\chi_j(t) = \chi(t - j)$. It follows that the metric 
\begin{equation*}
    g_j = \chi_j g + (1-\chi_j) \Tilde{g}_j
\end{equation*}
agrees with the original metric $g$ on the compact pieces $[0,j] \times T^3$ on each end, and extends to a metric on the compact manifold $\overline{X}$. 

We now compute the various curvature tensors of the metric $\Tilde{g}_j$, on the ends $T^3 \times [0,\infty)$. The Ricci curvature of $\Tilde{g}_j$ can be explicitly computed to be 
\begin{equation}\label{RicciComputation}
\begin{split}
    Ric_{\Tilde{g}_j} &= (\frac{-\partial_t^2 \phi_j}{\phi_j} - 2\frac{\partial^2_t \psi_j}{\psi_j}) dt^2 + (\frac{-\partial_t^2 \phi_j}{\phi_j} - 2\frac{\partial_t \phi_j}{\phi_j} \cdot \frac{\partial_t \psi_j}{\psi_j}) (\phi_j d\theta)^2  \\
    &+ (\frac{-\partial^2_t \psi_j}{\psi_j} - \frac{(\partial_t\psi_j)^2}{\psi_j^2} - \frac{\partial_t \phi_j}{\phi_j} \cdot \frac{\partial_t \psi_j}{\psi_j}) \psi_j^2 g_{T^2}
\end{split}
\end{equation}
By taking the trace, we may also compute the scalar curvature to be
\begin{equation}\label{ScalarComputation}
    s_{\Tilde{g}_j} = 2 \frac{-\partial_t^2 \phi_j}{\phi_j} - 4\frac{\partial^2_t \psi_j}{\psi_j} -4 \frac{\partial_t \phi_j}{\phi_j} \cdot \frac{\partial_t \psi_j}{\psi_j} - 2 \frac{(\partial_t\psi_j)^2}{\psi_j^2}
\end{equation}
Note here that by construction, $s_{\Tilde{g}_j}$ is uniformly bounded from below. The second fundamental form of metrics taking the form $dt^2 + g_t$ is given by $-\partial_t g_t/2$. For us, we can explicitly see it to be
\begin{equation}\label{SecondFundFormComputation}
    \mathrm{I\!I}_{\Tilde{g}_j} =\frac{(\partial_t\phi_j)^2}{\phi_j^2} \phi_j^2 d\theta^2 + \frac{(\partial_t\psi_j)^2}{\psi_j^2} \psi_j^2 g_{T^2}
\end{equation}
And the mean curvature can be computed as 
\begin{equation}\label{MeanCurvatureComputation}
    h_{\Tilde{g}_j} = \frac{-1}{2} ( \frac{(\partial_t\phi_j)^2}{\phi_j^2} + 2\frac{(\partial_t\psi_j)^2}{\psi_j^2} ) \ge \frac{3}{2}
\end{equation}
, where the lower bound follows from \ref{ApproxAssumptions}. By the hypothesis of $C^2$ closeness, for all $j$, the scalar curvatures of $(\overline{X},g_j)$ will be bounded from below by $s_{g_j} \ge C$ where $C$ is a uniform constant. Similarly, it follows that there is a uniform constant so that $Vol(\overline{X},g_j) \leq C \cdot Vol(X,g)$.

\section{$L^2$ Cohomology}\label{L2cohomology}
Suppose that $(X,g)$ is an oriented, noncompact $n$-manifold. Denote by $L^2\Omega^k(X,g)$ the $L^2$ completion of smooth $k$-forms with bounded $L^2$ norm taken with respect to $g$. If we restrict the deRham differential $d$ to a suitable dense linear subspace, we have a Hilbert complex
\begin{equation*}
    \cdots \rightarrow L^2\Omega^{k-1}(X,g)\rightarrow L^2\Omega^k(X,g) \rightarrow L^2\Omega^{k+1}(X,g) \rightarrow \cdots 
\end{equation*}
The maximal domain of existence for $d$ is 
\begin{equation*}
    Dom^k(d) = \{\alpha \in L^2\Omega^k(X,g) \vert d\alpha \in L^2\Omega^{k+1}(X,g)\}
\end{equation*}
Let $Z^k(X,g) = \{\alpha \in L^2\Omega^{k-1}(X,g) \vert d\alpha = 0\}$, and define the (reduced) $L^2$ cohomology to be
\begin{equation*}
    H^*_{(2)}(X,g) = Z^k(X,g)/\overline{dDom^{k-1}(d)}
\end{equation*}
Let $\mathcal{H}^k_{(2)}(X,g)$ denote the space of $g$-harmonic $k$-forms with finite $L^2$ norm. When $g$ is a complete metric, Gaffney \cite{Gaffney54} showed that there exists a Hodge-Kodaira decomposition
\begin{equation*}
    L^2\Omega^k(X,g) = \mathcal{H}^k_{(2)}(X,g) \oplus \overline{d\Omega^{k-1}} \oplus \overline{d^*\Omega^{k+1}}. 
\end{equation*}
This implies that there is an isomorphism
\begin{equation*}
     \mathcal{H}^*_{(2)}(X,g) \cong H^*_{(2)}(X,g)
\end{equation*}
When $X$ is orientable since the Poincar\'{e} duality isomorphism is an isometry, we also have Poincar\'{e} duality via the Hodge star
\begin{equation*}
    H^k_{(2)}(X,g) \cong H^{n-k}_{(2)}(X,g).
\end{equation*}
Therefore, when $n = 2k$ it makes sense to talk about $L^2$ self-dual and anti-self-dual forms. For hyperbolic cusps on $2n$ dimensional manifolds of maximal rank and finite volume, we have a computation due to Zucker \cite{Zucker82}. See also \cite{MazzeoPhillips90}.
\begin{theorem}[Zucker]
Suppose that $(X^{2n},g)$ is a finite volume complete orientable $2n$ dimensional manifold with asymptotically hyperbolic cusps. Then the $L^2$ cohomology is isomorphic to the various homology groups as follows:
\begin{equation*}
    H^k_{(2)}(X,g) \cong 
    \begin{cases}
        H^k_c(X) & k > n\\
        Im(H^k_c(X) \rightarrow H^k(X)) & k = n\\
        H^k(X) & k < n
    \end{cases}
\end{equation*}
\end{theorem}
The idea behind the proof is to take a Fourier decomposition of harmonic forms on the cusps with respect to harmonic forms on the flat cross sections; after analyzing the decay rates, one can show its possible to define a suitable retraction operator to define the isomorphism. Since $L^2$ cohomology is invariant under quasi-isometry, it follows that the same is true for asymptotically hyperbolic cusps. When $n = 4$, $k=2$ and $X$ has asymptotically hyperbolic $T^3$ cusps, it follows that
\begin{equation*}
    H^2_{(2)}(X,g) = Im(H^2_c(X) \rightarrow H^2(X)).
\end{equation*}
When $X = \overline{X} - L$, there are two relative long exact sequences of pairs given by $(\overline{X},L)$ and $(X,\partial X)$. Here $\partial X$ is the unit normal bundle of $L$, realizing the geometric boundary of $X$.
\begin{equation*}
\begin{tikzcd}
    H^2(\overline{X},L) \arrow{r} \arrow{d} & H^2(\overline{X}) \arrow{r} \arrow{d} & H^2(L) \arrow{d}\\
    H^2(X,\partial X) \arrow{r} & H^2(X) \arrow{r} & H^2(\partial X)
\end{tikzcd}
\end{equation*}
The arrows are the ones induced by restriction. A diagram chase shows that $\omega \in Im(H^2_c(X) \rightarrow H^2(X))$ if and only if for all the 2-tori $T_i \subset L$, $\langle \omega,[T_i] \rangle = 0$. Summarizing, this shows that:
\begin{proposition}
    Suppose $X = \overline{X} - L$ and $g$ is an asymptotically hyperbolic metric. Then the reduced $L^2$ cohomology is
    \begin{equation*}
    H^k_{(2)}(X,g) \cong 
    \begin{cases}
        H^k(\overline{X},L) & k > 2\\
        \{\omega \in H^2(X) \vert \langle \omega,[T_i] \rangle = 0\} & k = 2\\
        H^k(X) & k < 2
    \end{cases}
    \end{equation*}
\end{proposition}

\subsection{$L^2$ Chern-Weil theory}
Suppose that $(X,g)$ is a complete, finite volume 4 manifold. We extend Chern-Weil theory into the setting $L^2$ forms to define the $L^2$-Chern class of complex vector bundles. Suppose that $\mathcal{L}$ is a complex line bundle over $X$, and $A$ is a connection on $\mathcal{L}$ so that $F_A \in L^2\Omega^2(X,g)$. Then, we can define the $L^2$ Chern class of $L$ to be given by
\begin{equation*}
    c_1(\mathcal{L}) = \frac{i}{2\pi}[F_A]_{L^2}
\end{equation*}
On complete manifolds, the $L^2$ Chern class is invariant under $L^2_1$ perturbations of the underlying connection. More precisely, if $\alpha \in L^2_1\Omega^2(X,g)$, then it follows that 
\begin{equation*}
    [F_A]_{L^2} = [F_{A + a}]_{L^2}.
\end{equation*}
In particular, we have the equality
\begin{equation*}
    \begin{split}
        \int_X F_{A + a} \wedge F_{A + a} &= \int_X F_A \wedge F_A + da \wedge (2 F_A + da)\\
        &= \int_X F_A \wedge F_A
    \end{split}
\end{equation*}
where the second equality follows by considering a sequence $a_n \in C^\infty_c \Omega^1(X)$ with $a_n \xrightarrow{L^2_1} a$. Convergence and Stokes theorem on a finite truncation imply that
\begin{equation*}
    \int_X da \wedge (2 F_A + da) = \lim_{n\rightarrow\infty }\int_X da_n \wedge (2 F_A + da) = 0 
\end{equation*}
In fact we can do this more explicitly using the assumption that $c_1(\mathfrak{s})[T_i] = 0$. We simply choose a base connection $A_0$ so that $F_{A_0}\lvert_{\nu \Sigma} = 0$ on some tubular neighborhood of $\Sigma$. This ensures that there is a uniform bound $C$ so that $\| F_{A_0}\|_{g_j} \le C$ for all metrics $g_j$. It then follows that for all $spin^c$ connections $A$ on $\mathfrak{s}$ that differ from $A_0$ by an $L^2_1$ form,
\begin{equation*}
    \int_{X} F_A \wedge F_A = \int_X F_{A_0} \wedge F_{A_0}
\end{equation*}
so the self pairing of the $L^2$ Chern class of $\mathfrak{s}$ on $X$ agrees with the self pairing of the standard Chern class of $\mathfrak{s}$ on $\overline{X}$. By the Hodge-Kodaira decomposition, 
\begin{lemma}
    Suppose that $\omega \in L^2\Omega^2(X,g)$, and let $\gamma \in \mathcal{H}^2(X,g)$ be the harmonic representative for the class $[\omega]$. Then
    \begin{equation*}
        \int_X \lvert \omega \rvert^2 \ge \int_X \lvert \gamma \rvert^2
    \end{equation*}
    Let $\omega^+$ and $\gamma^+$ denote the self-dual components. Then we also have that
    \begin{equation*}
        \int_X \lvert \omega^+ \rvert^2 \ge \int_X \lvert \gamma^+ \rvert^2
    \end{equation*}
\end{lemma}

\section{Apriori Estimates on 1-forms} \label{1forms}
In this section we derive various $L^2$ estimates on 1-forms satisfying gauge fixing conditions that will be important later in performing the Seiberg-Witten bootstrap. We follow the approach of Biquard and di Cerbo. The first step is a uniform Poincar\'{e} inequality for functions over the family of metrics $g_j$. We begin with the lemma
\begin{lemma}
    Suppose that $g = dt^2 + g_t$ on the half cylinder $[0,\infty) \times Y$, and suppose that for some constant  $h_0 > 0$, the mean curvature $h$ satisfies 
    \begin{equation*}
        h = \frac{-\partial_t dvol_{g_t}}{2 dvol_{g_t}} \ge h_0
    \end{equation*}
    Then, for any function $f$, we have 
    \begin{equation*}
        \frac{1}{h_0}\int_{[0,T] \times Y} \lvert \partial_t f \rvert^2 \ge h_0 \int_{[0,T] \times Y} \lvert f \rvert^2 + \int_{t = T} \lvert f \rvert^2 - \int_{t=0} \lvert f \rvert^2
    \end{equation*}
    for all $T \in [0,\infty)$. 
\end{lemma}
\begin{proof}
    We integrate by parts to obtain
    \begin{equation*}
        \int_{t = T} \lvert f \rvert^2 dvol_{g_0} - \int_{t=0} \lvert f \rvert^2 dvol_{g_T} = \int_0^T 2 f \partial_t f + f^2 \frac{\partial_t dvol_{g_t}}{dvol_{g_t}} dvol_{g_t}
    \end{equation*}
    By Cauchy-Schwarz, we have the pointwise inequality $2\lvert f \partial_t f \rvert \le h_0 f^2 + h_0^{-1} (\partial_t f\rvert^2$, and the result is immediate. 
\end{proof}
For the model metrics $\Tilde{g_j}$, we can more explicitly set $h_0 = \frac{3}{2}$ to get a uniform constant. And the family of metrics $g_j$ are quasi-isometric to the model metrics $\Tilde{g_j}$, so we can get a similar estimate by weakening the constant factor. 
\begin{corollary}
    For the family of metrics $(\overline{X},g_j)$, there exists a constant $c$, uniform in $j$ so that for all functions $f$ with $\int_{\overline{X}} f dvol_{g_j} =0$
    \begin{equation*}
        \int_{\overline{X}} \lvert df\rvert^2_{g_j} dvol_{g_j} \ge c \int_{\overline{X}} f^2 dvol_{g_j} 
    \end{equation*}
\end{corollary}

\begin{proof}
    We proceed by contradiction. Suppose that there exists a sequence of functions $f_j$ with 
    \begin{equation*}
        \int_{\overline{X}} f_j dvol_{g_j} = 0, \|f_j\|_{g_j} = 0,
    \end{equation*}
    and so that $\|df_j\|_{g_j} \rightarrow 0$. Then it follows that $f_j$ is bounded in $L^2_1$ and thus has a weak limit $f$. Note that $\|df\|_{g} = 0$, so $f$ must be the constant function.

    Let $K \subset X$ be a compact subset. Then we note that
    \begin{equation*}
        \int_{K} f dvol_{g}= \lim_{j\rightarrow\infty} \int_{K} f_j dvol_{g_j}= \lim_{j\rightarrow\infty} \int_{\overline{X} - K} f_j dvol_{g_j}
    \end{equation*}
    But 
    \begin{equation*}
        \left\lvert \int_{\overline{X} - K} f_j \right\rvert \le \|f_j\|_{L^2(g_j)}\sqrt{Vol(\overline{X}- K)}
    \end{equation*}
    It follows that 
    \begin{equation*}
        \left\lvert\int_{K} f dvol_{g}\right\rvert \le  \sqrt{Vol(\overline{X}- K)}
    \end{equation*}
    Since for all $\varepsilon$, we can find a $K$ so that $Vol(\overline{X}- K) \le \varepsilon$, it follows that $f = 0$. Now choose $K\subset X$ to be compact with $\{0\} \times Y \subset K$, and let $\chi$ be a bump function with compact support in $X$, with $\chi\lvert_K = 1$. Since $\chi f_j$ have compact support and are bounded in $L^2_1$, it follows by Rellich embedding that $\|\chi f_j\|^2_{L^2(g)} \rightarrow 0$. On the other hand, the lemma tells us that there is a constant $C$ so that
    \begin{equation*}
        \begin{split}
            C \int_X \lvert (1-\chi) f_j \rvert^2 dvol_{g_j} &\le \int_X \lvert d (1-\chi) f_j \rvert^2 dvol_{g_j}\\
            & \le \int_X \lvert d f_j \rvert^2 + \lvert (d\chi_j) f_j \rvert^2 dvol_{g_j}
        \end{split}
    \end{equation*}
    By assumption $\|df_j\|^2_{L^2(g_j)} \rightarrow0$ and $\| (d\chi_j) f_j \|^2_{L^2(g_j)} \rightarrow 0$ again by the Rellich lemma. In light of the decomposition $f_j = \chi f_j + (1-\chi)f_j$, this contradicts $\|f_j\|_{L^2(g_j)} = 1$.
\end{proof}
Our next goal is to prove a similar uniform Poincar\'{e} inequality for 1-forms. We first need to understand convergence of harmonic 1-forms. 
\begin{proposition}
    Let $[a]\in H^1_{dR}(\overline{X})$ be a cohomology class and let $\alpha_j \in \mathcal{H}^1(X,g_j)$ be the corresponding harmonic representatives. Then $\alpha_j$ converges in $C^\infty_C(X)$ to a harmonic form $\alpha \in L^2\Omega^1(X)$.
\end{proposition}
This proposition ensures that for any compact subset $K\subset X$, $L \subset H^1(\overline{X})$, there exists a uniform $L^2_k$ bound on $g_j$-harmonic representatives of cohomology classes lying in $L$, depending only on $k$, $K$, and $L$.
\begin{proof}
    Fix a representative $\beta$ for $[a]$ on $\overline{X}$, and write
    \begin{equation*}
        \alpha_j = \beta + df_j
    \end{equation*}
    for $f_j$ satisfying
    \begin{equation*}
        \begin{cases}
            \int_{\overline{X}} f_j = 0\\
            \Delta^{g_j} f_j = -d^{*_j} \beta
        \end{cases}
    \end{equation*}
    Via integration by parts, we have
    \begin{equation*}
        \| df_j \|^2_{L^2(g_j)} =  - \int_{\overline{X}} f_j d^{*_j}\beta dvol_{g_j} \le \|f_j\|_{L^2(g_j} \|d^{*_j}\beta\|_{L^2(g_j)}
    \end{equation*}
    By construction of the metrics $g_j$, the norms of $\|\beta\|_{L^2(g_j)}$ and $\|d^{*_j}\beta\|_{L^2(g_j)}$ are both uniformly bounded in $j$. By lemma 3.2, we have that
    \begin{equation*}
        c\|f_j\|_{L^2(g_j}^2 \le \| df_j \|^2_{L^2(g_j)} \le \|f_j\|_{L^2(g_j} \|d^{*_j}\beta\|_{L^2(g_j)}
    \end{equation*}
    Thus, it follows that we have bounds $\|f_j\|_{L^2(g_j} \le c^{-1} \|d^{*_j}\beta\|_{L^2(g_j)}$, and so by elliptic bootstrapping, we can find a $C^\infty_c$ convergent subsequence of $f_j$ with limit $f$ satisfying
    \begin{equation*}
        \begin{cases}
            \int_{\overline{X}} f = 0\\
            \Delta^{g} f = -d^{*} \beta
        \end{cases}
    \end{equation*}
    And it follows that $\alpha = \beta + df$ is our desired harmonic representative. 
\end{proof}
We also need to show a uniform Poincar\'{e} inequality on the cusp ends of the approximating metrics. 
\begin{lemma}
    For the model metrics $\tilde{g_j}$ on $T^3 \times [0,T_j)$, we have the estimate for all 1-forms $\alpha$ with support in $[t_1,t_2] \subset [0,T_j)$, 
    \begin{equation*}
        \int_{[t_1,t_2] \times Y} \lvert d\alpha \rvert^2 + \lvert d^{*_j}\alpha \rvert^2 dvol_{\Tilde{g}} \ge c \int_{[t_1,t_2] \times Y} \lvert \alpha \rvert^2dvol_{\Tilde{g}}
    \end{equation*}
\end{lemma}

\begin{proof}
We recall that by the Bochner formula for 1-forms we have
\begin{equation*}
     \int_{[t_1,t_2] \times Y} \lvert d\alpha \rvert^2 + \lvert d^{*_j}\alpha \rvert^2 dvol_{\Tilde{g}} =  \int_{[t_1,t_2] \times Y} \lvert \nabla \alpha \rvert^2 + Ric(\alpha,\alpha) dvol_{\Tilde{g}}
\end{equation*}
Writing $\alpha = \beta + fdt$, with $\beta \in \Omega^1(Y)$, we can decompose the term
\begin{equation*}
    \begin{split}
        \nabla \alpha &= \nabla_{\partial_t} \alpha + d_Yf \otimes dt + f\nabla\rvert_Y dt + \nabla\rvert_Y dt \beta\\
        &= \nabla_{\partial_t} \alpha + d_Yf \otimes dt - f \mathrm{I\!I}(\cdot,\cdot) + \mathrm{I\!I}(\beta,\cdot) \otimes dt
    \end{split}
\end{equation*}
Here $\mathrm{I\!I}$ denotes the second fundamental form of the slices $T^3 \times \{t\}$. It follows that
\begin{equation*}
    \begin{split}
        \int \lvert \nabla \alpha \rvert^2 &= \int \lvert \nabla_{\partial_t} \alpha \rvert^2 + \lvert d_Y f + \mathrm{I\!I}(\beta,\cdot) \rvert^2 + \lvert -f\mathrm{I\!I}(\cdot,\cdot) + \nabla^Y \beta \rvert^2\\
        &= \int \lvert \nabla_{\partial_t} \alpha \rvert^2 + \lvert d_Y f \rvert^2 + \lvert \mathrm{I\!I}(\beta,\cdot) \rvert^2 + \lvert f\mathrm{I\!I}(\cdot,\cdot)\rvert^2 + \lvert \nabla^Y \beta \rvert^2 + 2\langle d_Y f,\mathrm{I\!I}(\beta,\cdot) \rangle - 2\langle f\mathrm{I\!I}(\cdot,\cdot), \nabla^Y \beta \rangle \\.
    \end{split}
\end{equation*}
From \ref{MeanCurvatureComputation}, we can write down the mean curvature computation as
\begin{equation*}
    \frac{-1}{2}\partial_t g_t = \frac{-\partial_t \phi_j}{\phi_j}(\phi_j d\theta)^2 + \frac{-\partial_t \psi_j}{\psi_j} \psi_j^2 g_{T^2}.
\end{equation*}
$\mathrm{I\!I}$ commutes with $\nabla^Y$ because it is constant on the slices, so we can deduce that
\begin{equation*}
    \begin{split}
        \langle \nabla^Y f,\mathrm{I\!I}(\beta,\cdot) \rangle &= \langle f,(\nabla^Y)^* \mathrm{I\!I}(\beta,\cdot) \rangle\\
        &= \langle f,\mathrm{I\!I}(\nabla^Y \beta)\rangle
    \end{split}
\end{equation*}
and the 2 cross terms cancel out. By further decomposing $\alpha = fdt + f_1d\theta + \gamma$, we note that
\begin{equation*}
    \begin{split}
        \lvert \mathrm{I\!I}(\beta,\cdot) \rvert^2 &= \frac{(\partial_t \phi_j)^2}{\phi_j^2} \lvert f_1 d\theta \rvert^2 + \frac{(\partial_t \psi_j)^2}{\psi_j^2} \lvert \gamma\rvert^2\\
        \lvert f \mathrm{I\!I} \rvert^2 &= \left( \frac{(\partial_t \phi_j)^2}{\phi_j^2} + 2  \frac{(\partial_t \psi_j)^2}{\psi_j^2} \right) \lvert fdt \rvert^2
    \end{split}
\end{equation*}
By also plugging in the formula for the Ricci curvature \ref{RicciComputation} from section 2, we can finally deduce that
\begin{equation*}
    \begin{split}
        Ric(\alpha,\alpha) + \lvert \mathrm{I\!I}(\beta,\cdot) \rvert^2 + \lvert f \mathrm{I\!I} \rvert^2 &= \left(\frac{-\partial_t^2 \phi_j}{\phi_j} - 2\frac{\partial^2_t \psi_j}{\psi_j} + \frac{(\partial_t\phi_j)^2}{\phi_j^2} + 2\frac{(\partial_t\psi_j)^2}{\psi_j^2} \right) \lvert fdt \rvert^2 \\
        &+ \left( \frac{-\partial_t^2 \phi_j}{\phi_j} - 2\frac{\partial_t \phi_j}{\phi_j} \cdot \frac{\partial_t \psi_j}{\psi_j} + \frac{(\partial_t\phi_j)^2}{\phi_j^2} \right) \lvert f_1 d\theta\rvert^2\\
        &+ \left( \frac{-\partial^2_t \psi_j}{\psi_j} - \frac{\partial_t \phi_j}{\phi_j} \cdot \frac{\partial_t \psi_j}{\psi_j} \right) \lvert \gamma \rvert^2
    \end{split}
\end{equation*}
But by the metric approximations \ref{ApproxEstimates}, we know that each of the coefficients is bounded from below by $-2$. Thus, we have
\begin{equation*}
    Ric(\alpha,\alpha) + \lvert \mathrm{I\!I}(\beta,\cdot) \rvert^2 + \lvert f \mathrm{I\!I} \rvert^2 \ge -2 \lvert \alpha \rvert^2.
\end{equation*}
But now, because the mean curvature of $g_j$ is bounded from below by $\frac{3}{2}$, it follows that 
\begin{equation*}
    \int_{[t_1,t_2] \times Y} \lvert \nabla_{\partial_t} \alpha \rvert^2 \ge \frac{9}{4} \int_{[t_1,t_2] \times Y} \lvert \alpha \rvert^2
\end{equation*}
Thus, 
\begin{equation*}
    \int_{[t_1,t_2] \times Y} \lvert d\alpha \rvert^2 + \lvert d^{*_j}\alpha \rvert^2 dvol_{\Tilde{g}} \ge \frac{1}{4} \int_{[t_1,t_2] \times Y} \lvert \alpha \rvert^2
\end{equation*}
\end{proof}

Note that the situation here holds also for metrics that are $C^2$-close to such metrics, such as the metrics $g_j$ that we consider. This is because in local coordinates $d^*$ depends only on the metric $g$ and its first derivatives.
\begin{corollary}
    There exists $T > 0$ and $c>0$ so that for all $j \ge T$ and $[t_1,t_2] \subset [T,T_j)$, we have the uniform estimate for the metrics ${g_j}$
    \begin{equation*}
        \int_{[t_1,t_2] \times Y} \lvert d\alpha \rvert^2 + \lvert d^{*_j}\alpha \rvert^2 dvol_{\Tilde{g_j}} \ge c \int_{[t_1,t_2] \times Y} \lvert \alpha \rvert^2dvol_{\Tilde{g_j}}
    \end{equation*}
    For all 1-forms $\alpha$ with support in ${[t_1,t_2] \times Y}$. 
\end{corollary}
This allows us to prove the estimates: 
\begin{proposition}
    For all 1-forms $a$ orthogonal to the space of $g_j$ harmonic 1-forms, there exists a $c$ independent of $j$ so that 
    \begin{equation*}
        \int_{\overline{X}} \lvert da \rvert^2 + \lvert d^{*_{g_j}} a\rvert^2 dvol_{g_j} \ge c \int_{\overline{X}} \lvert a \rvert^2
    \end{equation*}
\end{proposition}

\begin{proof}
    Suppose to the contrary that there is a sequence $\alpha_j$ of 1-forms so that $\lvert \alpha_j\rvert_{L^2(g_j)} = 1$ and $\int_{\overline{X}}\lvert d\alpha_j\rvert^2 + \lvert d^{*_j}\alpha_j \rvert^2 dvol_{g_j} \rightarrow 0$.Then, we may choose a diagonal subsequence, so that in the $C^\infty_c$ topology, $\alpha_j \rightarrow \alpha$. By construction, $d\alpha = d^{*_g}\alpha = 0$, so $\alpha \in \mathcal{H}^1_g(X)$. But by our computation of the $L^2$ cohomology of $X$ and proposition 3.3, it follows that $\alpha = 0$. 

    By using corollary 3.6, it now follows from the same proof as lemma 3.2 that we have a contradiction.
\end{proof}

\section{The Seiberg-Witten Equations on Asymptotically Hyperbolic Manifolds}\label{SeibergWitten}
The nonexistence result uses uses estimates coming from Seiberg-Witten theory. Let $\overline{X}$ be a closed oriented 4-manifold with $b^+(\overline{X}) \ge 2$ and suppose that $\mathfrak{s}$ is a $spin^c$ structure. Suppose that there exists an irreducible solution $(A,\phi)$ to the Seiberg-Witten equations. The Weitzenb\"{o}ck formula tells us that
\begin{equation*}
    0 = D_A^2 \phi = \nabla_A^*\nabla_A \phi + \frac{s_g}{4}\phi + \frac{\lvert\phi\rvert^2}{4} \phi
\end{equation*}
where $s_g$ is the scalar curvature of $g$. Integration by parts combined with the observation that $\lvert F_A^+ \rvert = 8 \lvert \phi \rvert^2$ gives us the estimate 
\begin{equation*}
    \frac{1}{32\pi^2} \int_X s_g^2 dvol_g \ge (c_1^+(\mathfrak{s}))^2
\end{equation*}
first proved by LeBrun in \cite{LeBrun95}. To state the nonexistence result we first recall the notion of a monopole class \cite{Kronheimer99} \cite{LeBrunIshida}.
\begin{definition}
    Let $\overline{X}$ be a closed orientable manifold, and let $\alpha \in H^2(\overline{X})/\text{torsion}$. If there exists a $spin^c$ structure $\mathfrak{s}$ such that for any metric $g$ the Seiberg Witten equations on $\mathfrak{s}$ admits an irreducible solution, say that $\alpha$ is a {monopole class}.
\end{definition}
The goal of this section is the following nonexistence theorem for asymptotically hyperbolic manifolds.
\begin{theorem}\label{SWEstimate}
    Suppose that $\overline{X}$ is a closed oriented 4-manifold, and suppose that $\alpha$ is a monopole class. Suppose that $L \subset \overline{X}$ is a disjoint collection of smoothly embedded 2-tori, with each component having self-intersection 0. Then, if 
    \begin{equation*}
        2\chi(\overline{X}) - 3\lvert \sigma(\overline{X}) \rvert \le \frac{1}{3} \alpha^2
    \end{equation*}
    $X = \overline{X} - L$ does not admit any asymptotically hyperbolic Einstein metrics. 
\end{theorem} 

We first derive apriori estimates that we need for the Seiberg-Witten bootstrap. The strategy is entirely analogous to the standard bootstrap in \cite{MorganBook}. Let $\alpha$ be a monopole class, and let $\mathfrak{s}$ be the corresponding $spin^c$ structure. Note that by the adjunction inequality \cite{Kronheimer99}, $\langle c_1(\mathfrak{s}),[T_i]\rangle = 0$ for each embedded torus $T_i$. Fix a base connection $A_0$ so that $F_{A_0} = 0$ on a tubular neighborhood of each of the 2-tori $\nu(T_i)$.  
\begin{proposition}\label{SWAprioriEstimates}
    Suppose that we have a sequence of solutions $(A_j,\phi_j)$ of the unperturbed Seiberg Witten equations on $(\overline{X},g_j)$
    \begin{equation*}
        \begin{cases}
            D_{A_j} \phi_j = 0\\
            F_{A_j}^+ = 2 (\phi_j \phi_j^*)_0
        \end{cases}
    \end{equation*}
    Let $a_j = A_0 - A_j$. By gauge fixing, we may assume that $d^{*_{g_j}}a_j = 0$ and 
    \begin{equation*}
        \int_{\overline{X}} a_j \wedge \beta_i \in [0,2\pi]
    \end{equation*}
    where $\beta_i \in H^3_{dR}(\overline{X})$ forms a fixed orthonormal basis. Then there exist uniform constants $K_1,K_2,K_3$, independent of $j$ so that
    \begin{equation*}
        \| a_j \|_{L^2_2(g_j)} \le K_1, \|\psi_j\|_{L^2_1(g_j)} \le K_2, \|\psi_j \|_{L^\infty} \le K_3
    \end{equation*}
\end{proposition}

\begin{proof}
These bounds follow from the usual approach to apriori estimates for the Seiberg-Witten equations. We first derive the $L^\infty$ bound on $\phi_j$. Suppose that $x \in \overline{X}$ is a maximal point for $\lvert \phi_j \rvert^2$. Then $Re \langle \nabla_{A_j}^* \nabla_{A_j} \phi_j,\phi_j \rangle \ge 0$ at $x$. By the Weiztenb\"ock formula, 
\begin{equation*}
    0 = Re \langle \nabla_{A_j}^* \nabla_{A_j} \phi_j,\phi_j \rangle + \frac{1}{4}\lvert\phi_j\rvert^4 + \frac{1}{4} s_{g_j} \lvert \phi_j\rvert^2
\end{equation*}
we see that at the point $x$,
\begin{equation*}
    \lvert \phi_j \rvert^2 \le -s_{g_j}
\end{equation*}
By construction, $s_{g_j}$ is uniformly bounded. So, it follows that $\lvert \phi_j\rvert$ is uniformly bounded, independent of $j$. Since  $vol(\overline{X},g_j)$ is uniformly bounded in $j$ as well, it also follows that $\lvert \phi_j \rvert_{L^2}$ is uniformly bounded. By integrating the Weitzenb\"ock formula for the Seiberg Witten equations, we have that
\begin{equation*}
    0 = \int_{\overline{X}} \lvert \nabla_{A_j} \phi_j \rvert^2 + \frac{1}{4}s_{g_j}\lvert \phi_j \rvert^2 + \frac{1}{4}\lvert\phi_j\rvert^4
\end{equation*}
Since the terms $s_{g_j}\lvert \phi_j \rvert^2$ and $\lvert\phi_j\rvert^4$ are both bounded, it follows that $\lvert \nabla_{A_j} \phi_j \rvert^2 $ is also uniformly bounded independent of $j$. Therefore, we have a uniform bound on $\lvert \phi_j \rvert_{L^2_1}$.  

We now proceed to show the $L^2_2$ bounds on $a_j$. Write $a_j = b_j + c_j$, where $b_j \in (\mathcal{H}^1_{g_j})^\perp$ and $c_j \in \mathcal{H}^1_{g_j}$. Proposition 3.3 gives us a uniform $L^2_2$ bound on $c_j$ since $c_j$ is contained inside a compact set in $H^1(\overline{X})$. Since $\int_{\overline{X}}\lvert \nabla_{A_j} \phi_j \rvert^2  \ge 0$, by the Weitzenb\"ock formula, we have the inequality
\begin{equation*}
    \int_{\overline{X}} \lvert F_{A_j}^+ \rvert^2 = \frac{1}{8}\int_{\overline{X}} \lvert \phi_j \rvert^4 \le \frac{1}{8}\int_{\overline{X}} s_{g_j}^2 
\end{equation*}
The first equality is from the curvature equation in the Seiberg Witten equations. This gives us a bound for $\lvert F_{A_j}^+ \rvert_{L^2(g_j}$. Since $d^{+_j} b_j = \frac{1}{2}(F_{A_j}^+ - F_{A_0}^+)$ and since $\lvert F_{A_0}^+ \rvert_{L^2(g_j)}$ is uniformly bounded, it follows that we have a bound for $\lvert d^{+_j} b_j  \rvert_{L^2(g_j)}$. Note that $\| db_j \|_{L^2} = 2 \|d^{+_{j}}b_j\|_{L^2}$ and so it follows from corollary 3.7 that we have $L^2_1$ bounds on $b_j$. 

We note from taking the covariant derivative on the curvature equation that
\begin{equation*}
    \nabla F_{A_j}^+ = (\nabla \phi_j) \otimes \phi_j^* + \phi_j \otimes (\nabla \phi_j^*) - Re \langle \nabla \phi_j, \phi_j \rangle Id
\end{equation*}
Note that the $L^\infty$ bound and the $L^2_1$ bound on $\phi_j$ combine to give us an $L^2_1$ bound on $F_{A_j}^+$. We use the curvature Seiberg-Witten equation $d^+ b_j  = F_{A_j}^+ - F_{A_0}^+$ to see that this then gives us an $L^2_2$ bound on $b_j$. Combining the $L^2_2$ uniform bounds for $b_j$ and $c_j$, we get a uniform bound for $a_j$, as desired.
\end{proof}

By using these estimates, we are ready to use the Seiberg-Witten bootstrap to construct a solution to the Seiberg-Witten equations on $(X,g)$. 
\begin{proposition}\label{SWExistence}
    Suppose that $(A_j,\phi_j)$ are solutions to the Seiberg-Witten equations on $(\overline{X},g_j)$ 
    \begin{equation*}
        \begin{cases}
            D_{A_j} \phi_j = 0\\
            F_{A_j}^+ = 2 (\phi_j \phi_j^*)_0
        \end{cases}
    \end{equation*}
    Then there exists a subsequence that converges, up to gauge transformation, in the $C^\infty_c$ topology to a solution $(A,\phi)$ of the Seiberg Witten equations, with $A - A_0 \in L^2_1\Omega(X,g)$. In particular, it follows that 
    \begin{equation*}
        \int_X F_A \wedge F_A = \int_X F_{A_0} \wedge F_{A_0} = c_1^2(\mathfrak{s})
    \end{equation*}
\end{proposition}

\begin{proof}
Suppose that $(A_n,\phi_n)$ are solutions to the Seiberg-Witten equations on $(\overline{X},g_n)$ and let $a_n = A_0 - A_n$. We proceed via the Seiberg Witten bootstrap argument on the compact sets $X_n \subset X$ (where $g_n = g$), and finish by extracting a diagonal subsequence of the solutions. To run the bootstrap, we need:
\begin{enumerate}
    \item A uniform $L^\infty$ bound and a uniform $L^2_2$ bound on $\lvert \phi_n \rvert$
    \item A uniform $L^2_1$ bound $\|a_n\|_{L^2_1(g_n)}$ 
\end{enumerate}
These uniform bounds come from proposition \ref{SWExistence}. To start the bootstrap, by the first equation we note that
\begin{equation*}
    D_{A_n} \phi_n = -a_n \cdot \phi_n
\end{equation*}
By the uniform $L^2_2$ bound on $a_n$ and the uniform $L^\infty$ bound on $\phi_n$, it follows that we have an $L^4_0$ bound on $ D_{A_n} \phi_n$ by Sobolev multiplication. By elliptic regularity, this implies an $L^4_1$ bound on $\phi_n$. Repeating the previous step with the Sobolev multiplication $L^2_2 \otimes L^4_1 \rightarrow L^3_1$, it follows that we get uniform $L^3_2$ bounds on $\phi_n$. And repeating again with $L^2_2 \otimes L^3_2 \rightarrow L^2_2$, we get a uniform $L^2_3$ bound on $\phi_n$. 

The uniform part of the bootstrapping now follows from $L^2_k \otimes L^2_k \rightarrow L^2_k$ for $k \ge 3$. We use the curvature Seiberg-Witten equation
\begin{equation*}
    d^+ a_n  = \frac{1}{2}(\phi_n\phi_n^*)_0 - F_{A_0}^+
\end{equation*}
and the $L^2_k$ bounds on $\phi_n$ to acquire uniform $L^2_{k+1}$ bounds on $a_n$ (here we use the assumption of Coulomb gauge). And we use the Dirac equation to acquire $L^2_{k+1}$ bounds for $\phi_n$. The Sobolev embedding theorem then tells us that a diagonal subsequence converges in $C^\infty_c$. 
\end{proof}

The proof of theorem \ref{SWEstimate} now quickly follows.
\begin{proof}
    Let $g$ be any asymptotically hyperbolic metric on $X$. Let $\mathfrak{s}$ be a $spin^c$ structure with $c_1(\mathfrak{s}) = \alpha$. From proposition \ref{SWExistence}, it follows that there exists a solution $(A,\phi)$ to the Seiberg-Witten equations on $X$. Thus, it follows from theorem \ref{DiCerboEstimate} that
    \begin{equation*}
        \frac{1}{32\pi^2} \int_X s_g^2 dvol_g \ge \alpha^2.
    \end{equation*}
    If $g$ were to be Einstein, then by \ref{DaiWeiEstimate},
    \begin{equation*}
        2\chi(X) - 3\lvert \sigma \rvert \ge \frac{1}{96\pi^2 }\int_X s_g^2 dvol_g 
    \end{equation*}
    with equality if and only if $g$ is Calabi-Yau with respect to a complex structure $J$. Since asymptotically hyperbolic metrics cannot be Ricci flat, it follows that when  $ 2\chi(X) - 3\lvert \sigma \rvert \le \frac{1}{3} \alpha^2$, $g$ cannot possibly be Einstein. 
\end{proof}

\section{$Pin^-(2)$ Monopoles}\label{Pin2}
The $Pin^-(2)$ equations are a variation of the Seiberg-Witten equations defined for a double cover $\tilde{X} \rightarrow X$ by Nakamura in \cite{Nakamura13}. Let $\ell = \Tilde{X} \times_{\pm 1} \mathbb{Z}$ be the local system corresponding to the double cover. Note that local systems on $X$ with coefficient group $\mathbb{Z}$ are in bijective correspondence with double covers of $X$, so we sill sometimes use the two notions interchangeably. The main goal of this section is to prove an analogous version of the nonexistence theorem from the previous section.
\begin{theorem}\label{Pin2Estimate}
    Let $X$ be a closed oriented manifold and let $\tilde{X} \rightarrow X$ be a double cover with $b^+(X,\ell) \ge 2$. Suppose that $\omega \in H^2(\overline{X},\ell)$ is a $Pin^-(2)$ monopole basic class, and let $L \subset \overline{X}$ be a disjoint collection of smoothly embedded 2-tori, with each component having self-intersection 0. Then if 
    \begin{equation*}
        2\chi(\overline{X}) - 3\lvert \sigma(\overline{X}) \rvert \le \frac{1}{3} \omega^2
    \end{equation*}
    the noncompact manifold $X = \overline{X} - L$ does not admit any asymptotically hyperbolic Einstein metrics. 
\end{theorem}
We prove theorem \ref{MainTheorem} by combining this with a gluing theorem proved by Nakamura (theorem \ref{Pin2Gluing}) for $Pin^-(2)$ monopole invariants in \cite{Nakamura15}.

\subsection{The $Pin^-(2)$ Monopole Equations}
We briefly recall the setup of $Pin^-(2)$ monopoles. For a more thorough exposition, see section 3 of \cite{Nakamura13}. Define the Lie group $Pin^-(2) = U(1) \sqcup jU(1) \subset Sp(1)$, with a double covering map $\phi: Pin^-(2) \rightarrow O(2)$ that maps 
\begin{equation*}
    \phi(z) = z^2, \phi(j) = \begin{pmatrix} 1 & 0 \\ 0 & -1\end{pmatrix}.
\end{equation*}
The Lie group $Spin^{c-}(4)$ is defined by $Spin^{c-}(4) = Spin(4) \times_{\pm 1} Pin^-(2)$, so that there is an exact sequence of groups
\begin{equation*}
    1 \rightarrow \{\pm1\} \rightarrow Spin^{c-}(4) \rightarrow SO(4) \times O(2) \rightarrow 1
\end{equation*}
The double covering $Pin^-(2) \rightarrow U(1)$ also induces a double covering $Spin^{c-}(4) \rightarrow Spin^c(4)$. The spinor representations $\Delta^{\pm}: Spin(4) \rightarrow M_{4}(\mathbb{C})$ extend to representations of $Spin^{c-}(4)$ where $U(1)$ acts by complex multiplication and $j$ acts by conjugation. 

Let $X$ be a closed oriented 4-manifold and let $\pi: \Tilde{X} \rightarrow X$ be an unbranched double cover, and let $g$ be any Riemannian metric on $X$. A $Spin^{c-}$ structure $\mathfrak{s} = (P,s,t)$ on $X$, where $P$ is a principal $Spin^{c-}$ bundle over $X$, $s: \tilde{X} \rightarrow P/Spin^c(4)$ is an isomorphism of double covers of $X$ and $t: Fr(X) \rightarrow P/Pin^-(2)$ is an isomorphism of principal $SO(4)$ bundles. The $O(2)$ bundle $E = P/Spin(4)$ is the characteristic bundle of the $spin^{c-}$ structure. By fixing an $\ell$-orientation on $E$, we can define a corresponding $\ell$-coefficient Euler class of $E$
\begin{equation*}
    c_1(\mathfrak{s}) = c_1(E) \in H^2(X,\ell).
\end{equation*}
Define the spinor bundles 
\begin{equation*}
    S^{\pm}= P \times_{Spin^{c-}} \Delta^{\pm}
\end{equation*}
to be the rank 2 complex vector bundles associated to the spinor representations. We can now write down the $Pin^-(2)$ monopole equations. Suppose that a $spin^{c-}$ structure $\mathfrak{s}$ is given. Let $E$ be the corresponding principal $O(2)$ bundle, and let $\lambda = \Tilde{X} \times_{\pm1} \mathbb{R}$, and $\ell = \Tilde{X} \times_{\pm1} \mathbb{Z}$ be the corresponding $\mathbb{R}$ and $\mathbb{Z}$ coefficient systems associated to the double covering. An $O(2)$ connection $A$ on $E$ induces a Dirac operator $D_A: \Gamma(S^+) \rightarrow \Gamma(S^+)$ and there is a standard quadratic map $q: S^+ \rightarrow \Omega^+(X,i\lambda)$. Note that $F_A \in \Omega^2(X,i\lambda)$. The $Pin^-(2)$-monopole equations are
\begin{equation*}
    \begin{cases}
        D_A \phi& = 0\\
        F_A^+ &= \frac{1}{2} \rho(\phi) + i\mu
    \end{cases}
\end{equation*}
here $\mu \in \Omega^+(X,i\lambda)$ is a perturbation; when $\mu = 0$, the equations are called the unperturbed $Pin^-(2)$-monopole equations. Let $\mathcal{A}$ be the space of $O(2)$ connections on $E$ and define the configuration space $\mathcal{C} = \mathcal{A} \times \Gamma(S^+)$. We also define the space of irreducible configurations $\mathcal{C}^* = \mathcal{A} \times (\Gamma(S^+)\setminus \{0\})$. The gauge group $\mathcal{G} = \Gamma(\Tilde{X} \times_{\pm1} U(1))$ acts on $\mathcal{C}$ and preserves solutions to the $Pin^-(2)$-monopole equations. Let$\mathcal{B}^* = \mathcal{C}^* / \mathcal{G}$ be the moduli space of configurations and let $\mathcal{M}_{Pin^-(2)}(X,\mathfrak{s})\subset \mathcal{B}^*$ denote the moduli space of irreducible solutions. 

For a given $k \ge 3$, we work with the $L^2_{k}$ completion of $\mathcal{C}$, and the $L^2_{k+1}$ completion of the gauge group. Fix a reference $O(2)$ connection $A_0$. As with Seiberg-Witten theory when $b_+(X,\ell) \ge 2$, for a generic choice of perturbation the $Pin^-(2)$ moduli space will be a smooth compact $d$-dimensional submanifold of $\mathcal{B}^*$, wehre the dimension $d$ is given by
\begin{equation*}
    d= \frac{1}{4}(c_1^2(E) - \sigma(X)) - (b_0(X,\ell) - b_1(X,\ell) + b_+(X,\ell)).
\end{equation*}
Note that in contrast to the ordinary Seiberg-Witten equations, the moduli space is not always orientable. The $Pin^-(2)$ invariant of $(X,\mathfrak{s})$ is defined is defined as the fundamental class
\begin{equation*}
    [\mathcal{M}_{Pin^-(2)}(X,\mathfrak{s})] \in H_d(\mathcal{B}^*).
\end{equation*}
In the case where $d = 0$, this amounts to a mod 2 count of solutions up to guage of the $Pin^-(2)$ equations under a generic pertubation. In this case, let $SW_{Pin^-(2)} \in \mathbb{Z}/2\mathbb{Z}$ denote the mod 2 count.

Nakamura proved a gluing formula for $Pin^-(2)$ monopoles \cite{Nakamura15}. It can be viewed as an analog of the blow-up formula from Seiberg-Witten theory in the $Pin^-(2)$ setting. For any 3-manifold $Y$, let $\ell$ be the $\mathbb{Z}$ local system on $S^1 \times Y $ coming from the nontrivial local system on $S^1$. For $g \ge 1$, construct a $\mathbb{Z}$ local system $\ell$ on $S^2 \times \Sigma_g$ as follows. Let $\ell_{T^2}$ be any nontrivial $\mathbb{Z}$ local system on the 2-torus, and recall that $\Sigma_g = \#^g T^2$. Define $\ell_{\Sigma_g} = \#^g \ell_{T^2}$, and define $\ell$ on $S^2 \times \Sigma_g$ by taking the pullback of $\ell_{\Sigma_g}$ through the pullback map. In \cite{Nakamura15}, Nakamura proved the following gluing result: 
\begin{theorem}[Nakamura]\label{Pin2Gluing}
    Let $X_1$ be a closed oriented connected 4-manifold with $b_+(X_1) \ge 2$ and a $spin^c$ structure $\mathfrak{s}_1$ so that the formal dimension of the Seiberg-Witten moduli space is 0 and $SW(X_1,\mathfrak{s}_1) = 1 \pmod{2}$. Let $X_2 = {\#}_i (S^2 \times \Sigma_{g_i}) {\#}_j (S^1 \times Y_j)$ with the a local system $\ell$ defined by connect summing the local systems on $S^2 \times \Sigma_{g_i}$ and $S^1 \times Y_j$ defined about. Then for any  $spin^{c-}$ structure $\mathfrak{s}_2$ on $X_2$, we have
    \begin{equation*}
        SW_{Pin^-(2)}(X_1 \# X_2,\mathfrak{s}_1 \# \mathfrak{s}_2) = 1 \pmod{2} 
    \end{equation*}
\end{theorem}
The condition for $X_1$ is satisfied for symplectic manifolds when $\mathfrak{s}_1$ is taken to be the $spin^c$ structure induced by the symplectic form. The reason that it should be considered a $Pin^-(2)$ analog of the blow-up formula is because $b_+(X_2,\ell_{X_2}) = 0$. Note also that 
\begin{equation*}
    c_1^2(\mathfrak{s}_1) = c_1^2(\mathfrak{s}_1 \# \mathfrak{s}_2).
\end{equation*}

We now briefly discuss the strategy for constructing examples in \ref{MainTheorem}. The idea is to use the gluing formula \ref{Pin2SWExistence} and apply \ref{Pin2Estimate}. Let $X_1$ be a closed symplectic manifold with $b^+(X_1) \ge 2$, and let $\mathfrak{s}_1$ be the corresponding canonical $spin^c$ structure. Then it follows that $SW(X_1,\mathfrak{s}_1) = 1$. Pick $X_2$ as in \ref{Pin2SWExistence} so that 
\begin{equation*}
    -\chi(X_2) + 2 \ge \frac{1}{3} (2 \chi(X_1) + 3\sigma(X_1)).
\end{equation*}
Since $c_1^2(\mathfrak{s}_1) = 2\chi(X_1) + 3 \sigma(X_1)$, and $\sigma(X_2)= 0$, it follows that 
\begin{equation*}
    2\chi(X_1 \# X_2) - 3\lvert \sigma(X_1 \# X_2) \rvert \ge c_1^2(\mathfrak{s}_1) + 2 \chi(X_2) -4 \ge \frac{1}{3} c_1^2(\mathfrak{s}_1)
\end{equation*}
Since $c_1^2(\mathfrak{s}_1 \# \mathfrak{s}_2) = c_1^2(\mathfrak{s}_1)$, it follows that we can apply $\ref{Pin2Estimate}$ and deduce \ref{MainTheorem}.

\subsection{Relation with the Seiberg-Witten equations} 
The $Pin^-(2)$ equations on $X$ are related to the standard $U(1)$ Seiberg-Witten equations on $\Tilde{X}$. As with the previous section, for a more detailed exposition, see section 3(iii) of \cite{Nakamura13}. Given a $spin^{c-}$ structure $\mathfrak{s}$ on $X$, there exists a corresponding $spin^c$ structure $\Tilde{\mathfrak{s}}$ on $\Tilde{X}$ defined as follows.
\begin{equation*}
    \begin{tikzcd}
        P \arrow["\Tilde{s}"]{r} \arrow{dr}& \Tilde{X}\arrow["\pi"]{d}\\
        & X
    \end{tikzcd}
\end{equation*}
The isomorphism $s: P/Spin^c \rightarrow \Tilde{X}$ gives us a map $\Tilde{s}: P \rightarrow \Tilde{X}$ that takes the structure of a principal $spin^c$ bundle over $\Tilde{X}$. Since the isomorphism $t$ lifts to an isomorphism $\tilde{t}: P/U(1) \rightarrow Fr(\Tilde{X})$, it follows that this defines a $spin^c$ structure $\Tilde{\mathfrak{s}}$ over $\Tilde{X}$. The characteristic $O(2)$ bundle $E$ over $X$ also lifts into the determinant $U(1)$ bundle $\Tilde{E}$ of $\Tilde{\mathfrak{s}}$ over $\Tilde{X}$. So it also follows that $2c_1^2(\mathfrak{s}) = c_1^2(\tilde{\mathfrak{s}})$. The spinor bundles $\Tilde{S}^{\pm}$ of $\Tilde{\mathfrak{s}}$ have a canonical identification $\pi^*S^{\pm}$, and a $Spin^{c-}$ connection $A$ has a canonical lift to a $Spin^c$ connection $\Tilde{A}$. 

Let $\iota: \Tilde{X} \rightarrow \Tilde{X}$ be the deck transformation of the double covering. Then, the $spin^c$ structure defined by $\iota \circ \Tilde{s}$ will be the conjugate $spin^c$ structure, and therefore $\iota$ is naturally covered by a principal $spin^c$-bundle map $\Tilde{\iota}: \Tilde{\mathfrak{s}} \rightarrow \overline{\Tilde{\mathfrak{s}}}$. Define 
\begin{equation*}
    J = [1,j^{-1}] \in Spin^{c-} = Spin(4) \times_{\pm 1} Pin^-(2).
\end{equation*}
Since $J$ is on the nonidentity component of $Spin^{c-}$, the right action of $J$ on $P \rightarrow X$ covers $\iota$, and it is not hard to see that $J$ is $\Tilde{\iota}$ composed with the complex conjugation map $\Tilde{\mathfrak{s}} \rightarrow \overline{\Tilde{\mathfrak{s}}}$.  

The $J$-action will induce actions on the spinor bundles $\Tilde{S}^\pm = P \times_{spin^c} \Delta^\pm$ by $I(p,\phi) = [pJ,J^{-1}\phi]$. This $I$-action is an antilinear involution of the spinor bundles, and
\begin{equation*}
    S^\pm \cong \Tilde{S}^\pm / I. 
\end{equation*}
Under this identifications, it follows that $\Gamma(S^\pm) \cong \Gamma(\Tilde{S}^\pm)^I$. Similarly, there is also an $I$-action on the space of $U(1)$ connections on the determinant line bundle $L$ of $\Tilde{\mathfrak{s}}$. This is because $J$ passes to an antilinear involution of the determinant line of $\Tilde{\mathfrak{s}}$. It follows that 
\begin{equation*}
    \mathcal{A}(E) = \mathcal{A}(L)^I,
\end{equation*}
where $\mathcal{A}(E)$ are the principal $O(2)$-connections on $E$ and $\mathcal{A}(L)$ are the principal $U(1)$-connections on $L$. Nakamura showed in \cite{Nakamura13} that solutions to the $Pin^-(2)$-monopole equations on $X$ are precisely the $I$-invariant solutions of the usual Seiberg-Witten solutions on $\Tilde{X}$ that are invariant under $J$.
\begin{proposition}[Nakamura]\label{Pin2SWExistence}
    Fix a Riemannian metric $g$ on $X$ and let $\Tilde{g}$ be the covering metric on $\Tilde{X}$. There is a bijective correspondence between the set of $Pin^-(2)$ monopoles on $(X,\mathfrak{s})$ and the set of $I$-invariant Seiberg-Witten monopoles on $(\Tilde{X},\Tilde{\mathfrak{s}})$. Furthermore, there is a canonical identification of moduli spaces
    \begin{equation*}
        \mathcal{M}_{Pin^-(2)}(X,\mathfrak{s}) \cong \mathcal{M}_{U(1)}(\Tilde{X},\Tilde{\mathfrak{s}})^I
    \end{equation*}
\end{proposition}
In particular, this implies that if $(X,\mathfrak{s})$ has nonvanishing $Pin^-(2)$ monopole invariant then given any metric $g$ on $X$, the pullback metric $\pi^*g$ on $\Tilde{X}$ will admit solutions to the unperturbed Seiberg-Witten equations for the $spin^c$ structure $\Tilde{\mathfrak{s}}$.

\subsection{Scalar Curvature Estimates}
We are now ready to prove theorem \ref{Pin2Estimate}. 
\begin{proof}
Let $\Tilde{L} = \pi^{-1}L$, and let $g$ be any asymptotically hyperbolic metric on $\overline{X} - L$. Let $g_j$ be a sequence of approximating metrics on $\overline{X}$ constructed in section \ref{MetricApprox}. Then the pullback metric $\Tilde{g} = \pi^*g$ will be an asymptotically hyperbolic metric on $\Tilde{X} - \Tilde{L}$, and the metrics $\Tilde{g}_j = \pi^*g_j$ will be approximating metrics on $\Tilde{X}$ for $\Tilde{X} - \Tilde{L}$. In particular, they will satisfy the same estimates \ref{RicciComputation} \ref{ScalarComputation} \ref{MeanCurvatureComputation} in section \ref{MetricApprox}. Our goal is now to show that there exists a solution $(A,\phi)$ to the unperturbed Seiberg-Witten equations on $\Tilde{X} - \Tilde{L}$ with $[F_A] = c_1(\Tilde{\mathfrak{s}}) \in H^2_{(2)}(\Tilde{X} - \Tilde{L})$. 

Since $\mathfrak{s}$ is a $Pin^-(2)$ basic class, it follows that there exists irreducible solutions $(A_j,\phi_j)$ to the unperturbed $Pin^-(2)$ monopole equations on $(X,g_j)$. As with the Seiberg-Witten case, fix a base connection $A_0$ with $F_{A_0}$ vanishing in a neighborhood of $L$; we can do this by the $Pin^-(2)$ adjunction inequality in \cite{Nakamura15}. By theorem \ref{Pin2SWExistence}, these solutions $(A_j,\phi_j)$ lift to solutions $(\Tilde{A}_j,\Tilde{\phi}_j)$ of the Seiberg-Witten equations on $(\Tilde{X},\Tilde{g}_j)$ on the $spin^c$ structure $\Tilde{\mathfrak{s}}$. It follows that we can apply theorem \ref{SWExistence} to $(\Tilde{A}_j,\Tilde{\phi}_j)$ to obtain a solution to the usual Seiberg-Witten equations on $(\Tilde{X} - \Tilde{L},\Tilde{g})$. In particular, we have the estimate:
\begin{equation*}
    \int_{\Tilde{X} - \Tilde{L}} s_{\Tilde{g}}^2 dvol_{\Tilde{g}} \ge \frac{1}{3} c_1^2(\Tilde{\mathfrak{s}}).
\end{equation*}
Since both these quantities are multiplicative under covers, we get corresponding estimates
\begin{equation*}
    \int_X s_g^2 dvol_g \ge \frac{1}{3} c_1^2(\mathfrak{s}).
\end{equation*}
It follows as in the proof of theorem \ref{SWEstimate} that if
\begin{equation*}
        2\chi(\overline{X}) - 3\lvert \sigma(\overline{X}) \rvert \le \frac{1}{3} c_1^2(\mathfrak{s})
\end{equation*}
then $X$ will not admit any asymptotically hyperbolic Einstein metrics.
\end{proof}
\bibliography{refs}

\end{document}